\documentclass[11pt]{amsart}
\usepackage[margin=1.15in]{geometry}
\usepackage{amscd,amssymb, amsmath, wasysym}
\usepackage{graphicx}
\usepackage{amsfonts}
\usepackage{mathrsfs}    
\usepackage{amsmath}    
\usepackage{amsthm}     
\usepackage{amscd}      
\usepackage{amssymb}    
\usepackage{eucal}      
\usepackage{latexsym}   
\usepackage{graphicx}   
\usepackage{verbatim}   
\usepackage[all]{xy}     

\makeatletter

\newcounter{thmcounter}

\numberwithin{thmcounter}{section}
\numberwithin{equation}{thmcounter}

\newtheorem{theorem}[thmcounter]{Theorem}
\newtheorem{proposition}[thmcounter]{Proposition}

\newtheorem{corollary}[thmcounter]{Corollary}

\theoremstyle{definition}
\newtheorem{definition}[thmcounter]{Definition}

\newtheorem{remark}[thmcounter]{Remark}

\newtheoremstyle{claim}{9pt}{3pt}{}{\parindent}{\bf}{.}{1em}{}

\theoremstyle{claim}

\newenvironment{namelist}[1]{%
\begin{list}{}
{
\settowidth{\labelwidth}{#1}%
\setlength{\labelsep}{0.3em}%
\setlength{\leftmargin}{\labelwidth}%
\addtolength{\leftmargin}{\labelsep}}}{%
\end{list}}

\newcommand{\nQ}{\mathbb{Q}}                     
\newcommand{\nP}{\mathbb{P}}                     

\newcommand{\nA}{\mathbb{A}}                     

\newcommand{\sO}{\mathscr{O}}                    

\newcommand{\sI}{\mathscr{I}}                    

\newcommand{\mf}[1]{\mathfrak{#1}}

\DeclareMathOperator{\Bl}{Bl}                    

\DeclareMathOperator{\codim}{codim}              

\DeclareMathOperator{\Fitt}{Fitt}                

\DeclareMathOperator{\Jac}{Jac}                  

\DeclareMathOperator{\reg}{reg}                  
\DeclareMathOperator{\Exc}{Exc}                  

\DeclareMathOperator{\rank}{rank}                

\newcounter{rkcounter}             
\begin{document}

\title[Mather-Jacobian multiplier ideals]{Vanishing theorems for Mather-Jacobian multiplier ideals on a Gorenstein projective variety}
\author{Wenbo Niu}

\address{Department of Mathematical Sciences, University of Arkansas, Fayetteville, AR 72701, USA}
\email{wenboniu@uark.edu}




\begin{abstract} In this paper, we establish several results related to vanishing theorems for Mather-Jacobian multiplier ideals on a Gorenstein projective variety, including an injectivity theorem, a Nadel-type vanishing theorem, a Griffith-type vanishing theorem for vector bundles, and a vanishing theorem for the asymptotic Mather-Jacobian multiplier ideals.
\end{abstract}
\maketitle

\section{Introduction}
\noindent The purpose of this paper is to establish vanishing theorems for Mather-Jacobian multiplier ideals on a Gorenstein projective variety. The Kodaira Vanishing Theorem as well as its generalization Kawamata-Viehweg Vanishing Theorem play fundamental roles in algebraic geometry. Based on that, Nadel Vanishing Theorem can be proved for classical multiplier ideals on a nonsingular projective variety. In the last few years, the theory of multiplier ideals has been generalized on an arbitrary projective variety in a series of works by de Fernex, Docampo, Ein, Ishii, Musta{\c{t}}{\u{a}} (\cite{Ein:MultIdeaMatherDis}, \cite{Ein:SingMJDiscrapency}, \cite{Roi:JDiscrepancy}). This new notion of multiplier ideal is called the Mather-Jacobian multiplier ideal, or MJ-multiplier ideal for short. It does not require the normality of the variety but still has many similar properties as the classical multiplier ideals, such as local vanishing, subadditivity, etc. However, it was not clear what kind of Nadel-type vanishing theorem could be established for MJ-multiplier ideals, which is certainly one of the most interesting properties that one should hope for. In this paper, we are trying to discuss this problem by proving various vanishing theorems on a Gorenstein projective variety.

Let $X$ be a Gorenstein projective variety and $\omega_X$ be the invertible dualizing sheaf of $X$. We first prove the following injectivity theorem for MJ-mulitplier ideals (for details, see Theorem \ref{thm:02}).
$$H^i(X,\omega_X\otimes L\otimes \widehat{\sI}(\mf{a}^t))\hookrightarrow H^i(X,\omega_X\otimes L\otimes \sO_X(G)\otimes \widehat{\sI}(\mf{a}^t)), \quad \mbox{ for }i\geq 0.$$
(In the introductory part, we only outline the typical forms of the theorems without stating too much on the technical conditions.) The injectivity theorem for classical multiplier ideals was proved by Ein-Popa \cite{Ein:GlobalDivInj}.  These  are all multiplier-ideal-type variants of the injectivity theorem of Koll{\'a}r and Esnault-Viehweg (\cite{Kollar:HighDireDualSheaf}, \cite{Viehweg:LctVanishing}).

Based on the injectivity theorem, we are able to deduce a Nadel-type vanishing theorem for MJ-multiplier ideals in the following form (see Corollary \ref{thm:01} for details)
$$H^i(X,\omega_X\otimes L\otimes \widehat{\sI}(\mf{a}^t))=0,\quad\mbox{ for }i>0.$$
With a little more efforts, we can further deduce a Griffith-type vanishing theorem for a vector bundle $E$ on $X$ as follows (see Theorem \ref{thm:04} for details)
$$H^i(X,\omega_X\otimes \det E\otimes S^mE\otimes \widehat{\sI}(\mf{a}^t))=0,\quad \mbox{ for } i>0,\ m\geq 0.$$
This Griffith-type vanishing theorem for classical multiplier ideals has also been discussed by Paoletti \cite{Paoletti:RmkMultiIdeals}.
Finally, we point out that we actually can use a more straightforward method to establish the above vanishing theorem without using the injectivity theorem. We illustrate this approach by proving a vanishing theorem for asymptotic MJ-multiplier ideals as follows (see Theorem \ref{thm:03} for details)
$$H^i(X,\omega_X\otimes \sO_X(mL+A)\otimes \widehat{\sI}(||mL||))=0,\quad\mbox{ for }i>0.$$
However, the approach via the injective theorem could shed a light in future on establishing vanishing theorems of MJ-multiplier ideals on arbitrary varieties.

It is worth mentioning that the assumption that $X$ is Gorenstein is essential in our paper. First of all, it makes the dualizing sheaf $\omega_X$ invertible so that we can work on a resolution of singularities by using the projection formula. Secondly, it provides the advantage that the Nash blowup can be represented as a usual blowup along an ideal. Then passing to a log resolution of the variety, the desired vanishing theorems can be deduced from the classical Kawamata-Viehweg vanishing theorem. Once we work on an arbitrary variety without the Gorenstein condition, all key techniques that we have used will fail.  Therefore it leads to us a very interesting question that what kind of vanishing theorems we could expect for MJ-multiplier ideals on arbitrary projective varieties.

\vspace{0.3cm}

\section{Mather-Jacobian multiplier ideals}
\noindent Throughout the paper, we work over an algebraically closed field $k$ of characteristic zero. By variety we mean a separated reduced irreducible scheme of finite type over $k$. On a projective variety $X$, we use {\em line bundles} and {\em (Cartier) divisors} interchangeably (cf, \cite[Proposition II.6.15.]{Har:AG}). Furthermore, the concept of $\nQ$-divisors works perfectly on $X$ and one can define {\em nef} ({\em big, ample, semiample,} etc.) $\nQ$-divisors (for details, see \cite{Lazarsfeld:PosAG1}). A $\nQ$-divisor is called {\em $\nQ$-effective} if it is $\nQ$-linearly equivalent to an effective divisor. A nef $\nQ$-divisor is called {\em abundant} if its Iitaka dimension is equal to its numerical dimension (see, for example, \cite{Ein:GlobalDivInj}) and we discuss this in Definition \ref{def:02}. We also use $\omega_X$ to denote the dualizing sheaf of $X$ when $X$ is Cohen-Macaulay.\\

We recall the definition of Mather-Jacobian multiplier ideals. For more detailed discussion and properties, we refer to the paper \cite{Ein:MultIdeaMatherDis}.

\begin{definition} Let $X$ be a variety of dimension $n$ and $f:Y\longrightarrow X$ be a resolution of singularities factoring through the Nash blowup of $X$. Then the image of the canonical homomorphism
$$f^*(\wedge^n\Omega^1_X)\longrightarrow \wedge^n\Omega^1_Y$$
is an invertible sheaf of the form $\Jac_f\cdot\wedge^n\Omega^1_X$ where $\Jac_f$ is the relative Jacobian ideal. The effective divisor $\widehat{K}_{Y/X}$ defined by $\Jac_f$ is called the {\em Mather discrepancy divisor}.
\end{definition}

\begin{definition} Let $X$ be a variety and $\mf{a}\subseteq \sO_X$ a nonzero ideal on $X$. Given a log resolution $f:Y\longrightarrow X$ of $\Jac_X\cdot\mf{a}$ such that $\mf{a}\cdot\sO_Y=\sO_Y(-Z)$ and $\Jac_X\cdot\sO_Y=\sO_Y(-J_{Y/X})$ for some effective divisors $Z$ and $J_{Y/X}$ on $Y$ (such resolution automatically factors through the Nash blowup, see Remark 2.3 of \cite{Ein:MultIdeaMatherDis}). The {\em Mather-Jacobian multiplier ideal} of $\mf{a}$ of exponent $t\in \nQ_{\geq 0}$ is defined by
$$\widehat{\sI}(X,\mf{a}^t)=f_*\sO_Y(\widehat{K}_{Y/X}-J_{Y/X}-\lfloor tZ\rfloor),$$
where $\lfloor\ \ \rfloor$ means the round down of an $\nQ$-divisor. Sometimes we simply write it as $\widehat{\sI}(\mf{a}^t)$ and call it as {\em MJ-multiplier ideal}.
\end{definition}

\begin{definition}\label{def:01} Let $X$ be a Gorenstein variety of dimension $n$. Consider the fundamental class of $X$,
$$c_X:\wedge^n\Omega^1_X\longrightarrow \omega_X.$$
The image of $c_X$ is $\mf{J}_1\otimes \omega_X$, where $\mf{J}_1$ is called the {\em ideal of the first Nash subscheme} of $X$. The {\em lci-defect ideal} of $X$ is defined to be $\mf{d}_X=(\Jac_X:\mf{J}_1)$ where $\Jac_X$ is the Jacobian ideal of $X$.
\end{definition}

\begin{remark} The ideal $\mf{d}_X$ defines a subscheme of $X$ supported on the locus of non local complete intersection points of $X$. This follows from the fact that when $X$ is a local complete intersection, the image of $c_X$ is $\Jac_X\otimes\omega_X$ and therefore in this case $\mf{d}_X=\sO_X$. The ideal $\mf{d}_X$ was also defined and studied in \cite{Roi:JDiscrepancy}.

An important property of the ideal $\mf{J}_1$ is that we can use it to represent the Nash blowup of $X$. More precisely, consider the projection morphism
$$\pi:\nP(\wedge^n\Omega^1_X)\longrightarrow X$$
which is an isomorphism over the nonsingular locus $X_{\reg}$ of $X$. The {\em Nash blowup} $\widehat{X}$ is the closure of $\pi^{-1}(X_{\reg})$ in $\nP(\wedge^n\Omega^1_X)$ with the reduced scheme structure. Restricting the tautological bundle of $\nP(\wedge^n\Omega^1_X)$ to $\widehat{X}$, we obtain the {\em Mather canonical divisor} $\sO_{\widehat{X}}(\widehat{K})$ on $\widehat{X}$. Now since the morphism $c_X$ has image $\mf{J}_1\otimes \omega_X$ and $\omega_X$ is a line bundle, it induces an isomorphism from $\widehat{X}$ to the blowup of $X$ along $\mf{J}_1$, i.e.
$$\widehat{X}\simeq \Bl_{\mf{J}_1}X.$$
Furthermore, by tracing the tautological bundles, we see that
$$\mf{J}_1\cdot\sO_{\widehat{X}}\simeq\sO_{\widehat{X}}(\widehat{K})\otimes \pi^*\omega^{-1}_X.$$
This is the main reason that we need to assume $X$ is Gorenstein throughout this paper.
\end{remark}

The following proposition was proved for normal $\nQ$-Gorenstein varieties in \cite[Corollary 9.3]{Ein:JetSch}. We follow the same idea to prove it for Gorenstein varieties without assuming the normality.

\begin{proposition}\label{p:01} Let $X$ be a Gorenstein variety. Then one has
$$\Jac_X=\mf{d}_X\cdot \mf{J}_1,$$
where $\Jac_X$ is the Jacobian ideal, $\mf{J}_1$ is the ideal of the first Nash subscheme, and $\mf{d}_X$ is the lci-defect ideal.
\end{proposition}
\begin{proof} It suffices to prove $\Jac_X\subseteq \mf{d}_X\cdot \mf{J}_1$. The question is local, so we can assume that $X$ is affine in $\nA^N$ and is defined by an ideal $I_X$.  Let $n=\dim X$ and $c=\codim_{\nA^N}X$ and assume that $I_X:=(F_1,F_2,\cdots,F_r)$. We can take the generators $F_i$'s general such that any $c$ of them will define a general complete intersection containing $X$. More precisely, let $J\subset \{1,2,\cdots, r\}$ with $|J|=c$, let $I_J$ be the ideal generated by $F_i$'s with $i\in J$, and let $V_J$ be the subscheme of $\nA^N$ defined by $I_J$. If we take $F_i$'s general, then each $V_J$ is a complete intersection in $\nA^N$ and $V_J=X$ at the generic point of $X$. Set $q_J=(I_J:I_X)\cdot\sO_X$ and $\Jac_J=\Jac_{V_J}\cdot \sO_X$ where $\Jac_{V_J}$ is the Jacobian ideal of $V_J$. Write $\omega_J$ to be the dualizing sheaf of $V_J$. Consider the following canonical morphisms
$$\wedge^n\Omega^1_X\stackrel{c_X}{\longrightarrow}\omega_X\stackrel{u_J}{\longrightarrow}\omega_J|_X.$$
It has been proved in \cite[Proposition 9.1]{Ein:JetSch} that $u_J$ is an injective and
$$u_J(\omega_X)=q_J\otimes \omega_J|_X, \quad u_J(\mf{J}_1\otimes \omega_X)=\Jac_J\otimes \omega_J
|_X.$$
Hence we get an equality $\mf{J}_1\cdot q_J=\Jac_J$. Set $\mf{b}=\sum_J q_J$ and note that $\Jac_X=\sum_J\Jac_J$, we then deduce that $\mf{J}_1\cdot \mf{b}=\Jac_X$.
Finally, since $\Jac_X=\mf{J}_1\cdot \mf{b}\subseteq\mf{J}_1\cdot \mf{d}_X\subseteq\Jac_X$, the result follows immediately.
\end{proof}

The first part of the following proposition is well-know but we still include it here for the convenience of the reader.
\begin{proposition}\label{q:01} Let $f:Y\longrightarrow X$ be a smooth morphism of varieties.
\begin{itemize}
\item [(1)] Let $\Jac_X$ and $\Jac_Y$ be the Jacobian ideals of $X$ and $Y$ respectively, then
$$\Jac_Y=\Jac_X\cdot \sO_Y.$$
\item [(2)] Assume that $X$ and $Y$ are Gorenstein varieties and let $\mf{J}^X_1$ and $\mf{J}^Y_1$ be the ideals of the first Nash subschemes of $X$ and $Y$ respectively, then
$$\mf{J}^Y_1=\mf{J}^X_1\cdot \sO_Y.$$
Furthermore, let $\mf{d}_X$ and $\mf{d}_Y$ be the lci-defect ideals of $X$ and $Y$ respectively, then
$$\mf{d}_Y=\mf{d}_X\cdot \sO_Y.$$
\end{itemize}
\end{proposition}
\begin{proof} (1) Since $f$ is smooth, we have a short exact sequence
$$0\longrightarrow f^*\Omega^1_X\longrightarrow \Omega^1_Y\longrightarrow \Omega^1_{Y/X}\longrightarrow 0,$$
where $\Omega^1_{Y/X}$ is a locally free sheaf of rank $e$. The above sequence is locally splitting, so working locally at each closed point, we can assume
$$\Omega^1_Y=f^*\Omega^1_X\oplus \Omega^1_{Y/X}.$$
Now since $\Omega^1_{Y/X}$ is locally free of rank $e$, it is easy to see that for any $i\geq 0$,
$$\Fitt^{i+e}\Omega^1_Y=\Fitt^i f^*\Omega^1_X=(\Fitt^i\Omega^1_X)\cdot\sO_Y.$$
Take $i=\dim X$, then we get the desired result.

(2)  We prove the equality $\mf{J}^Y_1=\mf{J}^X_1\cdot \sO_Y$ first. Let $\omega_X$ and $\omega_Y$ be the dualizing sheaves of $X$ and $Y$. Note that they are all invertible sheaves since the varieties are Gorenstein. Assume first that the morphism $f$ is \'{e}tale. Then we have the identities  $$f^*\Omega^1_X=\Omega^1_Y, \quad \mbox{ and } f^*\omega_X=\omega_Y.$$ The result then follows immediately by the definition of $\mf{J}_1$.

For the general case, locally around any closed point of $Y$, the morphism $f$ can be factored as $U\longrightarrow X\times_k\nA^r\longrightarrow X$, where the first morphism is an \'{e}tale cover and the second one is the projection. Therefore it suffices to prove the case that $Y=X\times_k\nA^r$ for $r\geq 1$, and the morphism $f:Y\longrightarrow X$ is the projection. We can further assume that $X$ is affine, embedded in an affine space $X\subset \nA^N$, and defined by an ideal $I_X$. Therefore, $Y$ can be assumed to be embedded in the affine space $\nA^{N+r}$ and defined by $I_Y$. We write $p$ and $q$ as the natural projections of $\nA^{N+r}=\nA^N\times\nA^r$ to $\nA^N$ and $\nA^r$, respectively. Note that $f=p|_Y$.

Write $d=\dim X$ and $c=\codim_{\nA^N}X$. We can choose $c$ general equations $f_1,\cdots, f_c\in I_X$ such that the ideal $I_M=(f_1,\cdots, f_c)$ defines a complete intersection $M$ in $\nA^N$ and $M=X$ at the generic point of $X$. Let $I_{X'}=(I_M:I_X)$ be the ideal defining the residual part $X'$ of $X$ in $M$ and write $I_{X'/X}=I_{X'}\cdot\sO_X$. Then there are following canonical morphisms (\cite[Appendix 9.]{Ein:JetSch})
$$\wedge^d\Omega^1_X\stackrel{c_X}{\longrightarrow}\omega_X\stackrel{u}{\longrightarrow}\omega_M|_X$$
in which the morphism $u$ is injective and makes $\omega_X$ as a $\sO_X$-submodule of $\omega_M|_X$. Under the morphism $u$, we have $u(\omega_X)=I_{X'/X}\otimes \omega_M|_X$  and $u(\mf{J}^X_1\otimes \omega_X)=\Jac_M\cdot\sO_X\otimes \omega_M|_X$, where $\Jac_M$ is the Jacobian ideal of $M$. Hence we get the equality
\begin{equation}\label{eq:05}
\mf{J}^X_1\cdot I_{X'/X}=\Jac_M\cdot\sO_X.
\end{equation}

Since $p$ is the projection, the ideal $I_{M'}=I_M\cdot \sO_{\nA^{N+r}}$ also defines a complete intersection $M'$ in $\nA^{N+r}$ such that $M'=Y$ at the generic point of $Y$. Let $I_{Y'}=(I_{M'}:I_Y)$ be the ideal of the residual part $Y'$ of $Y$ in $M'$ and write $I_{Y'/Y}=I_{Y'}\cdot\sO_Y$. Notice that $I_{Y'}=I_{X'}\cdot\sO_{\nA^{N+r}}$ and therefore
\begin{equation}\label{eq:07}
I_{X'/X}\cdot\sO_Y=I_{Y'/Y}
\end{equation} Exactly as in the case of $X$, we have canonical morphisms
$$\wedge^{d+r}\Omega^1_Y\stackrel{c_Y}{\longrightarrow}\omega_Y\stackrel{v}{\longrightarrow}\omega_{M'}|_Y$$
where $v$ is an injection and makes $\omega_Y$ as a $\sO_Y$-submodule of $\omega_{M'}|_Y$. We obtain that $v(\omega_Y)=I_{Y'/Y}\otimes \omega_{M'}|_Y$ and $v(\mf{J}^Y_1\otimes\omega_Y)=\Jac_{M'}\cdot\sO_Y\otimes \omega_{M'}|_Y$, where $\Jac_{M'}$ is the Jacobian ideal of $M'$. Thus we get the equality
\begin{equation}\label{eq:06}
\mf{J}^Y_1\cdot I_{Y'/Y}=\Jac_{M'}\cdot\sO_Y.
\end{equation}

Now we pull back the equality (\ref{eq:05}) and notice that $\Jac_{M'}\cdot\sO_Y=(\Jac_M\cdot\sO_X)\cdot\sO_Y$ by Proposition \ref{q:01} as well as the equality (\ref{eq:07}) so we obtain
$(\mf{J}^X_1\cdot\sO_Y)\cdot I_{Y'/Y}=\Jac_{M'}\cdot\sO_Y$.
Hence by using (\ref{eq:06}), we deduce that
$$(\mf{J}^X_1\cdot\sO_Y)\cdot I_{Y'/Y}=\mf{J}^Y_1\cdot I_{Y'/Y}.$$
This implies $(\mf{J}^X_1\cdot\sO_Y)=\mf{J}^Y_1$ because $Y$ is Gorenstein and $I_{Y'/Y}$ is an invertible sheaf.

For the equality $\mf{d}_Y=\mf{d}_X\cdot \sO_Y$, base on what we have proved, it is a direct consequence of the definition of lci-defect ideals and the fact that the morphism $f$ is flat.
\end{proof}

\begin{proposition}\label{p:02} Let $f:Y\longrightarrow X$ be a smooth morphism of varieties. Let $\mf{a}$ be an ideal of $\sO_X$ and $t\in \nQ_{\geq 0}$. Write $\mf{b}=\mf{a}\cdot\sO_Y$. Then one has
$$\widehat{\sI}(Y,\mf{b}^t)=\widehat{\sI}(X,\mf{a}^t)\cdot\sO_Y$$
\end{proposition}
\begin{proof} Take a log resolution of $\mf{a}\cdot \Jac_X$ as $\mu:X'\longrightarrow X$ such that $\mf{a}\cdot\sO_{X'}=\sO_{X'}(-Z)$ and $\Jac_X\cdot\sO_{X'}=\sO_{X'}(-J_{X'/X})$ for some effective divisors $Z$ and $J_{X'/X}$. We make a fiber product $Y'=X'\times_X Y$ and obtain the following diagram
$$\begin{CD}
Y'@>\nu>> Y\\
@VV f'V @VVf V\\
X'@>\mu>>X
\end{CD}$$
Since $f$ is smooth and by the construction of $\mu$, $Y'$ is nonsingular and $\nu:Y'\longrightarrow Y$ is a log resolution of $\mf{b}\cdot \Jac_Y$ and therefore $\mf{b}\cdot\sO_{Y'}=\sO_{Y'}(-Z')$ and $\Jac_Y\cdot\sO_{Y'}=\sO_{Y'}(-J_{Y'/Y})$ for some effective divisors $Z'$ and $J_{Y'/Y}$. Note that
$$f'^*Z=Z'\quad \mbox{and }\quad f'^*J_{X'/X}=J_{Y'/Y}.$$
In addition, by the base change, we have $\Omega^1_{Y'/Y}=f'^*\Omega^1_{X'/X}$ which implies $\Fitt^0(\Omega^1_{Y'/Y})=\Fitt^0(\Omega^1_{X'/X})\cdot\sO_{Y'}$. Hence we deduce that
$$f'^*(\widehat{K}_{X'/X})=\widehat{K}_{Y'/Y}.$$

Now write $$D=\widehat{K}_{X'/X}-J_{X'/X}-\lfloor tZ\rfloor\quad \mbox{ and }\quad D'=\widehat{K}_{Y'/Y}-J_{Y'/Y}-\lfloor tZ'\rfloor.$$
By our setting, we have $f'^*D=D'$. By the definition, we have
$$\widehat{\sI}(Y,\mf{b}^t)=\nu_*\sO_{Y'}(D),\quad \mbox{ and }\widehat{\sI}(X,\mf{a}^t)=\mu_*\sO_{X'}(D).$$
But since $f$ is flat, by \cite[III.9.3.]{Har:AG}, $f^*\mu_*\sO_{X'}(D)=\nu_*f'^*\sO_{X'}(D)$. Finally by using $f'^*D=D'$ we get $f^*\mu_*\sO_{X'}(D)=\nu_*\sO_{Y'}(D')$ which proves the result.
\end{proof}

From the proof above, we can immediately deduce the following easy corollary. For the definition of MJ-singularities, we refer to the paper \cite{Ein:SingMJDiscrapency}.
\begin{corollary} Let $f:Y\longrightarrow X$ be a smooth morphism of varieties. Then $X$ is MJ-canonical (MJ-log canonical) if and only if so is $Y$.
\end{corollary}

In the last of this section, we mention the injectivity theorem on nonsingular projective varieties. We recall the definition of abundant divisors. A good reference that fits our purpose is \cite[Preliminaries]{Ein:GlobalDivInj}.

\begin{definition}\label{def:02} A nef $\nQ$-divisor $D$ on a projective normal variety $X$ is called {\em abundant} if its Iitaka dimension $\kappa(D)$ equals its numerical dimension $\nu(D)$ which is the largest integer $m$ such that $D^m\cdot Y\neq 0$ for a subvariety $Y$ of dimension $m$. A nef $\nQ$-divisor on a projective variety is called {\em abundant} if so is its pullback to the normalization of the variety.
\end{definition}

\begin{remark}
Typical abundant divisors include semiample divisors and nef and big divisors. Here we quote two properties of nef and abundant divisors proved in \cite[Lemma 2.3.]{Ein:GlobalDivInj}. Assume that $X$ is a normal projective variety and $B$ and $C$ are nef and abundant $\nQ$-divisors, then $B+C$ is nef and abundant. If $f:Y\longrightarrow X$ is a subjective morphism with $Y$ normal and projective, then $f^*B$ is nef and abundant.
\end{remark}
The following Injectivity Theorem was stated in \cite[Corollary 5.12(b)]{Viehweg:LctVanishing} which is an extension of Koll{\'a}r's injectivity theorem \cite{Kollar:HighDireDualSheaf}. Note that we allow the boundary divisor $\Delta$ to be zero. In fact, when $\Delta$ is nonzero the theorem is exactly \cite[Corollary 5.12(b)]{Viehweg:LctVanishing}. But when $\Delta$ is zero, the theorem is the same as \cite[Theorem 3.1]{Ein:GlobalDivInj} for the case $\mf{a}=\sO_X$.

\begin{theorem}[{\cite[Corollary 5.2 (b)]{Viehweg:LctVanishing}}]\label{p:03} Let $L$ be a line bundle on a smooth projective variety $X$, and let $\Delta=\sum_i \delta_i\Delta_i$ be a simple normal crossings divisor with $0\leq \delta_i<1$ for all $i$. Assume that $L-\Delta$ is nef and abundant and that $B$ is an effective divisor such that $L-\Delta-\varepsilon B$ is $\nQ$-effective for some $0<\varepsilon<1$. Then the natural morphisms
$$H^i(X,\sO_X(K_X+L))\longrightarrow H^i(X,\sO_X(K_X+L+B))$$
are injective for all $i\geq 0$.
\end{theorem}

\section{Vanishing theorems for Mather-Jacobian multiplier ideals}

\noindent In this section, we prove the main results of this paper. Out strategy is to establish an injectivity theorem (Theorem \ref{thm:02}) for MJ-multiplier ideals along the line of \cite{Ein:GlobalDivInj}. Then the vanishing theorem for MJ-multiplier ideals (Corollary \ref{thm:01}) is a direct consequence of the injectivity theorem. However, we can actually prove the vanishing theorem in a more straightforward technique and we will show it in the proof of the vanishing theorem for the asymptotic MJ-multiplier ideals (Theorem \ref{thm:03}). We hope that using injectivity theorem could shed light on building more general theorems for MJ-multiplier ideals on arbitrary varieties.

\begin{theorem}\label{thm:02} Let $X$ be a Gorenstein projective variety and $\mf{d}_X$ be the lci-defect ideal of $X$. Let $\mf{a}\subseteq\sO_X$ be an ideal and $t\in
\nQ_{\geq 0}$. Assume that $A$, $B$ and $L$ are line bundles and $G$ is an effective divisor such that $\mf{a}\otimes A$ and $\mf{d}_X\otimes B$ are globally generated, $L-tA-B$ is nef and abundant, and $L-tA-B-\varepsilon G$ is $\nQ$-effective for some $0<\varepsilon<1$. Then the natural morphisms
$$H^i(X,\omega_X\otimes L\otimes \widehat{\sI}(\mf{a}^t))\longrightarrow H^i(X,\omega_X\otimes L\otimes \sO_X(G)\otimes \widehat{\sI}(\mf{a}^t))$$
are injective for all $i\geq 0$.
\end{theorem}
\begin{proof} Write $n=\dim X$. Consider the fundamental class $c_X:\wedge^n\Omega^1_X\longrightarrow \omega_X$, which has the image $\mf{J}_1\otimes \omega_X$ where $\mf{J}_1$ is the ideal of the first Nash subscheme. So we have a surjective morphism
\begin{equation}\label{eq:01}
\wedge^n\Omega^1_X\longrightarrow \mf{J}_1\otimes\omega_X\longrightarrow 0.
\end{equation}
Since $\omega_X$ is a line bundle, we can identify the Nash blowup $\widehat{X}$ as the blowup of $X$ along the ideal $\mf{J}_1$. Furthermore, the surjective morphism (\ref{eq:01}) induces the diagram
$$\xymatrix{
\widehat{X}\simeq \Bl_{\mf{J}_1}X  \ar@{^{(}->}[r] &\nP(\wedge^n\Omega^1_X\otimes\omega^{-1}_X)\simeq \nP(\wedge^n\Omega^1_X)\ar[d]^{\pi}  \\
  &X & }$$
where $\pi: \nP(\wedge^n\Omega^1_X)\longrightarrow X$ is the natural projection. Denote by $\nu=\pi|_{\widehat{X}}:\widehat{X}\longrightarrow X$. Thus by tracing the tautological bundles in the diagram above, we deduce that on the Nash blowup $\widehat{X}$,
\begin{equation}\label{eq:02}
\mf{J}_1\cdot\sO_{\widehat{X}}\simeq\sO_{\widehat{X}}(\widehat{K})\otimes \nu^*\omega^{-1}_X,
\end{equation}
where $\widehat{K}$ is the Mather canonical divisor.

Let $\Jac_X$ be the Jacobian ideal of $X$. Consider a log resolution $f:X'\longrightarrow X$ of $\Jac_X\cdot\mf{a}\cdot \mf{J}_1\cdot \mf{d}_X$ and $G$ such that $$\Jac_X\cdot \sO_{X'}(-J_{X'/X}),\ \mf{a}\cdot\sO_{X'}=\sO_{X'}(-Z), \ \mf{J}_1\cdot\sO_{X'}=\sO_{X'}(-J_1), \mbox{ and } \mf{d}_X\cdot\sO_{X'}=\sO_{X'}(-D_X)$$
where $J_{X'/X},\ Z,\ J_1, \mbox{ and } D_X$ are effective divisors and the supports of them together with the supports of $f^*G$ and $\Exc (f)$ are simple normal crossings. By Proposition \ref{p:01}, on $X'$ we have an equality of effective divisors
\begin{equation}\label{eq:03}
J_{X'/X}=J_1+D_X.
\end{equation}
The morphism $f$ then factor through the Nash blowup $\widehat{X}$ as in the following diagram
$$\xymatrix{
X' \ar[rr]^{\widehat{f}}\ar[dr]_{f} && \widehat{X} \ar[dl]^{\nu} \\
              &X           &
}$$
and we have $\widehat{K}_{X'/X}\simeq K_{X'}-\widehat{f}^*(\widehat{K})$ (see \cite[1.6, 1.7]{Ein:DivValArc}), where $K_{X'}$ is the canonical divisor of $X'$. By using the identity (\ref{eq:02}), we obtain that
$$\widehat{K}_{X'/X}\simeq K_{X'}-\widehat{f}^*(\widehat{K})\simeq K_{X'}+J_1-f^*\omega_X.$$
Therefore, from the equality (\ref{eq:03}) we deduce that
\begin{equation}\label{eq:04}
\widehat{K}_{X'/X}-J_{X'/X}-\lfloor tZ\rfloor\simeq K_{X'}-D_X-\lfloor tZ\rfloor-f^*\omega_X.
\end{equation}
Now by assumption that $\mf{a}\otimes A$ and $\mf{d}_X\otimes B$ are globally generated, the line bundles
$f^*A-Z$ and $f^*B-D_X$ are all nef and abundant. Since $L-tA-B$ is nef and abundant, so is its birational pullback $f^*L-tf^*A-f^*B$.
Hence the divisor
$$f^*L-D_X-tZ=(f^*L-tf^*A-f^*B)+(f^*B-D_X)+(tf^*A-tZ)$$
is nef and abundant. Also, by assumption, we have $f^*L-tf^*A-f^*B-\varepsilon f^*G$ is $\nQ$-effective for some $0<\varepsilon<1$. Set
$$L'=\lceil f^*L-D_X-tZ\rceil,  \quad\mbox{and } \Delta=tZ-\lfloor tZ\rfloor.$$ Then $\Delta$ is a simple normal crossing divisor and $\lfloor \Delta\rfloor=0$. Note also that $L'-\Delta$ is nef and abundant and $L'-\Delta-\varepsilon f^*G$ is $\nQ$-effective. Thus by Theorem \ref{p:03}, the natural morphisms
\begin{equation}\label{eq:08}
H^i(X',\sO_{X'}(K_{X'}+L'))\longrightarrow H^i(X',\sO_{X'}(K_{X'}+L'+f^*G))
\end{equation}
are injective for all $i\geq 0$.
But the equation (\ref{eq:04}) tells us that
$$K_{X'}+\lceil f^*L-D_X-tZ\rceil)\simeq\widehat{K}_{X'/X}-J_{X'/X}-\lfloor tZ\rfloor+f^*\omega_X+f^*L$$
and in addition by the local vanishing theorem for MJ-multiplier ideals (\cite{Ein:MultIdeaMatherDis}),
$$R^if_*\sO_{X'}(K_{X'}+\lceil f^*L-D_X-tZ\rceil)=0,\mbox{ for }i>0.$$
Therefore, by the definition that $\widehat{\sI}(\mf{a}^t)=f_*\sO_{X'}(\widehat{K}_{X'/X}-J_{X'/X}-\lfloor tZ\rfloor)$ and the projection formula, we can identify
$$H^i(X,\omega_X\otimes L\otimes \widehat{\sI}(\mf{a}^t))=H^i(X', K_{X'}+\lceil f^*L-D_X-tZ\rceil),$$
$$H^i(X,\omega_X\otimes L\otimes \sO_X(G)\otimes \widehat{\sI}(\mf{a}^t))=H^i(X',\sO_{X'}(K_{X'}+L'+f^*G)).$$
Thus the theorem follows from the injectivity of (\ref{eq:08}) immediately.
\end{proof}
\begin{corollary}\label{thm:01} Let $X$ be a Gorenstein projective variety and $\mf{d}_X$ be the lci-defect ideal of $X$. Let $\mf{a}\subseteq\sO_X$ be an ideal and $t\in\nQ_{\geq 0}$. Assume that $A$, $B$ and $L$ are line bundles such that $\mf{a}\otimes A$ and $\mf{d}_X\otimes B$ are all globally generated and $L-tA-B$ is nef and big. Then one has
$$H^i(X,\omega_X\otimes L\otimes \widehat{\sI}(\mf{a}^t))=0,\quad\mbox{ for }i>0.$$
\end{corollary}
\begin{proof} We can take a sufficient ample effective divisor $G$ on $X$ such that it eliminates all higher cohomology groups of $\omega_X\otimes L\otimes \sO_X(G)\otimes \widehat{\sI}(\mf{a}^t)$ by Serre's vanishing theorem. Since $L-tA-B$ is nef and big, it is nef and abundant and there exists a small number $0<\varepsilon<1$ such that $L-tA-B-\varepsilon G$ is $\nQ$-effective. Then we just apply the theorem above.
\end{proof}

\begin{corollary} Let $X$ be a local complete intersection projective variety and $\mf{a}\subseteq\sO_X$ be an ideal. Suppose that $A$ and $L$ are line bundles and $t\in\nQ_{\geq 0}$  such that $A\otimes \mf{a}$ is globally generated and $L-tA$ is nef and big. Then
$$H^i(X, \omega_X\otimes L\otimes\widehat{\sI}(\mf{a}^t))=0,\quad\mbox{ for }i>0.$$
\end{corollary}
\begin{proof} Simply notice that if $X$ is a local complete intersection then its lci-defect ideal $\mf{d}_X$ is trivial and the previous corollary can be applied.
\end{proof}

We prove a Griffith-type vanishing theorem for a vector bundle. For the meaning that a vector bundled is twisted by a $\nQ$-divisor, we refer to the book \cite{Lazarsfeld:PosAG2}.

\begin{theorem}\label{thm:04} Let $X$ be a Gorenstein projective variety and $E$ a vector bundle of rank $\geq 2$. Let $\mf{a}$ be an ideal of $\sO_X$ and $t\in\nQ_{\geq0}$. Assume that $A$ and $B$ are line bundles such that $\mf{a}\otimes A$ and $\mf{d}_X\otimes B$ are all globally generated. If one of the following two conditions holds:
\begin{itemize}
\item [(1)] $E$ is nef and big and $E(-tA-B)$ is nef;
\item [(2)] $E$ is nef and $E(-tA-B)$ is nef and big,
\end{itemize}
then one has
$$H^i(X,\omega_X\otimes \det E\otimes S^mE\otimes \widehat{\sI}(\mf{a}^t))=0,\quad \mbox{ for  }i>0,\ m\geq 0.$$
\end{theorem}
\begin{proof} Write $r=\rank E$. Let $Y=\nP(E)$ with the natural projection $f:Y\longrightarrow X$ and let $L=\sO_{Y}(1)$ be the tautological line bundle of $Y$. Note that $f$ is a smooth morphism and $Y$ is also Gorenstein. Set $\mf{b}=\mf{a}\cdot\sO_Y$.

By Proposition \ref{p:01} and Proposition \ref{q:01}, we have
\begin{equation}\label{eq:10}
\Jac_X=\mf{d}_X\cdot\mf{J}^X_1,\mbox{ and }\Jac_Y=\mf{d}_Y\cdot\mf{J}^Y_1
\end{equation}
where $\Jac_Y=\Jac_X\cdot\sO_Y$, $\mf{J}^Y_1=\mf{J}^X_1\cdot\sO_Y$ and $\mf{d}_Y=\mf{d}_X\cdot\sO_Y$.
Note that $\mf{b}\otimes f^*A$ and $\mf{d}\otimes f^*B$ are all globally generated by the assumption. The conditions (1) and (2) means that either $L$ is nef and big and $L-tf^*A-f^*B$ is nef, or $L$ is nef and $L-tf^*A-f^*B$ is nef and big. Thus applying  Corollary \ref{thm:01} to $Y$, it is easy to establish vanishing
$$H^i(Y,\omega_Y\otimes L^m\otimes \widehat{\sI}(\mf{b}^t))=0,\quad \mbox{ for }i>0,\ m\geq 2.$$
The duality theory tells us that
$\omega_Y=f^*\omega_X\otimes \det \Omega^1_{Y/X}$.
In addition, on $Y$ there is an short exact sequence
$$0\longrightarrow \Omega^1_{Y/X}(1)\longrightarrow f^*E\longrightarrow \sO_Y(1)\longrightarrow 0,$$
from which  we obtain
$f^*\det E=\det \Omega^1_{Y/X}\otimes \sO_Y(r)$,
and therefore
$$\omega_Y=f^*(\omega_X\otimes\det E)(-r). $$
Furthermore, by Proposition \ref{p:02}, $\widehat{\sI}(Y,\mf{b}^t)=f^*\widehat{\sI}(X,\mf{a}^t)$. So we have
$$\omega_Y\otimes L^m\otimes \sI(\mf{b}^t)=f^*(\omega_X\otimes\det E\otimes \widehat{\sI}(X,\mf{a}^t))(m-r).$$
The result follows if we take $m\geq r$ and use the projection formula.
\end{proof}

In the last of this section, we prove a vanishing theorem for the asymptotic construction of MJ-multiplier ideals. For the classical asymptotic multiplier ideals on nonsingular varieties we refer to the book \cite{Lazarsfeld:PosAG2}. Let us briefly recall the definition of asymptotic MJ-multiplier ideals associated to a line bundle. More detailed discussion and application can be found in the paper \cite{Niu:Nullstel}.

Given a line bundle $L$ on a variety $X$. Write $\mf{b}_k$ to be the base ideal of the linear system $|kL|$ for an integer $k\geq 1$. Then the collection of ideals $\{\mf{b}_k\}_{k\geq 1}$ is a graded family of ideals satisfying $\mf{b}_i\cdot\mf{b}_j\subseteq \mf{b}_{i+j}$ for all $i,j\geq 1$. (More examples of graded family can be found in \cite[Example 1.2]{Ein:UniBdSymPw}.) Fix an index $m\geq 1$ and consider the set of MJ-multiplier ideals
\begin{equation*}
\{\widehat{\sI}(X,\mf{b}_{mp}^{\frac{1}{p}})\}_{p\geq 1}
\end{equation*}
It is standard to check that there is a unique maximal ideal in the above set. This ideal is defined to be the {\em asymptotic Mather-Jacobian multiplier ideal at level $m$ associated to $L$} and is denoted by $\widehat{\sI}(||mL||)$. A basic fact that plays the essential role in the asymptotic multiplier ideal theory is that for $p\gg0$, we actually have
$\widehat{\sI}(||mL||)=\widehat{\sI}(X,\mf{b}_{mp}^{\frac{1}{p}}).$

\begin{theorem}\label{thm:03} Let $X$ be a Gorenstein projective variety and $\mf{d}_X$ be the lci-defect ideal of $X$. Let $A$, $B$ and $L$ be line bundles on $X$ such that $\mf{d}_X\otimes B$ is globally generated and $L$ has non-negative Iitaka dimension. Suppose that  one of the following two conditions holds
\begin{itemize}
\item [(1)] $A-B$ is nef and big;
\item [(2)] $A-B$ is nef and $L$ is big.
\end{itemize}
Then one has
$$H^i(X,\omega_X\otimes \sO_X(mL+A)\otimes \widehat{\sI}(||mL||))=0,\mbox{ for }i>0.$$
\end{theorem}
\begin{proof}Taking $p$ sufficiently large, we have $$\widehat{\sI}(||mL||)=\widehat{\sI}(\mf{b}_{pm}^{\frac{1}{p}}),$$ where $\mf{b}_{pm}$ is the base ideal of the system $|pmL|$. Note that $\mf{b}_{mp}\otimes L^{mp}$ is globally generated. If assume the condition (1), then we just apply Corollary \ref{thm:01} directly to get the desired vanishing result. Now we assume the condition (2) for the rest of the proof.

Take a log resolution $f:X'\longrightarrow X$ of $\Jac_X\cdot\mf{b}_{mp}\cdot \mf{J}_1\cdot \mf{d}_X$, where $\mf{J}_1$ is the ideal of the first Nash subscheme of $X$, such that $$\Jac_X\cdot \sO_{X'}(-J_{X'/X}),\ \mf{b}_{mp}\cdot\sO_{X'}=\sO_{X'}(-F), \ \mf{J}_1\cdot\sO_{X'}=\sO_{X'}(-J_1), \mbox{ and } \mf{d}_X\cdot\sO_{X'}=\sO_{X'}(-D_X)$$
where $J_{X'/X},\ F,\ J_1, \mbox{ and } D_X$ are effective divisors and the supports of them together with the support of $\Exc (f)$ are simple normal crossings. As we have discussed in the proof of Theorem \ref{thm:02}, such log resolution factors through the Nash blowup of $X$. Furthermore, we have
$$J_{X'/X}=J_1+D_X, \mbox{ and } \widehat{K}_{X'/X}\simeq K_{X'}+J_1-f^*\omega_X.$$
Hence we deduce that
$$\widehat{K}_{X'/X}-J_{X'/X}-\frac{1}{p}F\simeq K_{X'}-f^*\omega_X-D_X-\frac{1}{p}F.$$
On $X'$, we define a line bundle $$L'=f^*(mpL)\otimes \sO_{X'}(-F).$$
Since $mpL$ is big and $\mf{b}_{mp}$ is the base ideal of $mpL$, we obtain that $L'$ is big and globally generated. Now we can write
\begin{eqnarray*}
 & &f^*(\omega_X+mL+A)+\widehat{K}_{X'/X}-J_{X'/X}-\frac{1}{p}F\\
    &\simeq &f^*(mL)+f^*A+K_{X'}-D_X-\frac{1}{p}F\\
    &\simeq & K_{X'}+\frac{1}{p}(f^*(pmL)-F)+f^*(A-B)+f^*B-D_X.
\end{eqnarray*}
Notice that in the last expression, we have $f^*(pmL)-F=L'$ is nef and big and $f^*(A-B)$ and $f^*B-D_X$ are all nef. Hence we apply the Kawamata-Viehweg vanishing theorem to the last expression to deduce that
$$H^i(X',f^*(\omega_X+mL+A)+\widehat{K}_{X'/X}-J_{X'/X}-\lfloor\frac{1}{p}F\rfloor)=0, \quad \mbox{ for }i>0.$$
Finally, using the projection formula and the definition of MJ-multiplier ideals and the local vanishing property, we deduce the desired vanishing
$$H^i(X,\omega_X\otimes\sO_X(mL+A)\otimes \widehat{\sI}(\mf{b}^{\frac{1}{p}}_{mp}))=0, \quad \mbox{ for }i>0.$$

\end{proof}
\bibliographystyle{alpha}

\end{document}